\documentclass[%
 reprint,
 amsmath,amssymb,an
 aps,
 onecolumn,
]{revtex4-2}

\usepackage{amsmath}
\usepackage{amssymb}
\usepackage{hyperref}
\usepackage{amsthm}
\usepackage{enumerate}
\usepackage{array}
\usepackage{breqn}
\usepackage{tabularx}
\usepackage{physics}
\usepackage{tikz}
\usepackage{tabularx}
\usepackage{multirow}
\usepackage{todonotes}
\usepackage[margin=1.4in]{geometry}

\theoremstyle{plain}
\newtheorem{theorem}{Theorem}[section]
\newtheorem{corollary}[theorem]{Corollary}
\newtheorem{lemma}[theorem]{Lemma}

\newtheorem{proposition}[theorem]{Proposition}

\theoremstyle{definition}
\newtheorem{defn}[theorem]{Definition}
\newtheorem{example}[theorem]{Example}
\newtheorem{conjecture}[theorem]{Conjecture}

\theoremstyle{remark}
\newtheorem{remark}[theorem]{Remark}

\newcommand{\qend}{\operatorname{QEnd}}
\newcommand{\qhom}{\operatorname{QHom}}
\newcommand{\AHom}{\operatorname{AHom}}
\newcommand{\Az}{\operatorname{Az}}

\begin{document}

\title{\LARGE{The Eudoxus Reals} \\
\small{PROMYS 2023 Research Lab}}

\author{AJ Kumar}
\author{Reese Long}
\author{Andrew Tung}
\author{Ivan Wong}
\date{August 6, 2023}

\begin{abstract}
    We examine a unique construction of the real numbers which proceeds directly from the integers using approximately linear-endomorphisms with finite error, called near-endomorphisms. In this paper, we show that the set of near-endomorphisms forms a complete ordered field isomorphic to the reals. Moreover, we show that there are uncountably many near-endomorphisms without reference to the reals. We then investigate a natural extension of near-endomorphisms, which we call quasi-homomorphisms, to other abelian groups. Extending prior results about the construction of the $p$-adic numbers and the rational adele ring, we find the ring of near-endomorphisms of certain localizations of the integers, and suggest further directions for exploration.
\end{abstract}

\maketitle

\section{Background}

The most common construction of $\mathbb{R}$ begins from $\mathbb Z$, constructs $\mathbb Q$ as the field of fractions of $\mathbb Z$, and $\mathbb R$ as the completion of $\mathbb Q$ with respect to the usual absolute value. The construction of the reals using the ring of near-endomorphisms has the benefit of avoiding the notion of completion.

\begin{defn}
    A function $f: \mathbb{Z} \to \mathbb{Z}$ is said to be a \textit{near-endomorphism} if there exists some constant $C$ such that $|f(a+b) - f(a) - f(b)| < C$ for all $a, b \in \mathbb{Z}$.
\end{defn}

Intuitively, a near-endomorphism is a almost-linear function over $\mathbb Z$, with some small variation.

\begin{defn}
    Define two near-endomorphisms $f$ and $g$ to be \textit{equivalent} under $\sim$ if $f-g$ is bounded.
\end{defn}

For example, $f+c$ for any constant $c$ is still a near-endomorphism, and is equivalent to $f$.

\begin{defn}
    Let $\mathbb{E}$ be the set of near-endomorphisms under the equivalence relation defined above.
\end{defn}

It turns out that $\mathbb{E} \cong \mathbb{R}$, which we prove later. One way to think about this intuitively is that for any real number $a$, there is a function $f_a : \mathbb{R} \to \mathbb{R}$ with $f_a(x) = ax$. Conversely, every equivalence class of near-endomorphisms corresponds to a slope, since a near-endomorphism is ``almost linear", and the idea that two functions are equivalent if their difference is bounded is intuitively because they have the same slope.

\begin{example}
    The function $g_{a}(x) = \lfloor ax \rfloor$ is a near-endomorphism. This near-endomorphism corresponds to the real number $a$, since its ``slope" is $a$.
\end{example}

It should be noted that not all seemingly linear functions $f$ satisfying $\lim_{x \to \infty} \frac{f(x)}{x} = a$ for some $a \in \mathbb{R}$ are near-endomorphisms.

\begin{example}
    Define the function $g(n) := \{(a,b) \in \mathbb{Z} : a^2 + b^2 \leq n\}$. It turns out that $\lim_{n \to \infty} \frac{g(n)}{n} = \pi$, by Pick's Theorem argument. Moreover, it approximately linear. However, we can find an arbitrarily high value of $g(x) - g(x-1) - g(1)$ if we choose $x$ to be the product of a large amount of $1 \pmod{4}$ primes, which means there is no $C$ where $|f(a+b) - f(a) - f(b)| < C$ for all $a, b$. So $g$ is not a near-endomorphism.
\end{example}

In Section \ref{iso}, we prove rigorously that the set of near-endomorphisms is isomorphic to the field of real numbers. In Section \ref{construct}, we show that $\mathbb{E}$ is an ordered field completely from scratch (i.e. without presupposing the existence of the reals). In Section \ref{uncountability} we prove that $\mathbb{E}$ is uncountable directly from the definition, again without assuming the existence of the reals. In Section \ref{generalize}, we generalize the idea of a near-endomorphism for any two abelian groups. This larger framework allows us to find the ring of quasi-endomorphisms for other groups, which we do for localizations of $\mathbb Z$ in Section \ref{local}.

\section{Isomorphism Between \texorpdfstring{$\mathbb{E}$}{E} and \texorpdfstring{$\mathbb{R}$}{R}}
\label{iso}

With these definitions and the intuition in mind, we can prove that there is an isomorphism between the set of equivalence classes of near-endomorphisms and the real numbers. 

We first prove a lemma that helps bound the values of an arbitrary near-endomorphism $f$:

\begin{lemma}
\label{linear bound}
    If $f$ is a near-endomorphism, then $|f(na) - nf(a)| < |n|C$ for all $n, a \in \mathbb{Z}$.
\end{lemma}

\begin{proof}
    For $n$ positive, we induct on $n$. For the base case of $n=1$, $|f(a)-f(a)| = 0$ which is clearly less than $C$. We inductively assume that $|f(na) - nf(a)| < nC$ for some $n \in \mathbb{N}$. Then, by the definition of a near-endomorphism, we know that
    \begin{align*}
        |f(na+a) - f(na) - f(a)| &< C \\
        |f((n+1)a) - f(na) - f(a)| &< C \\
        |f((n+1)a) - (n+1)f(a)| &< (n+1)C
    \end{align*}
    by triangle inequality. Hence, the inductive hypothesis holds for the $n+1$ case. Using a similar argument for negative integers $n$, we can show that $|f(na) - nf(a)| < |n|C$.
\end{proof}

\begin{corollary}
\label{mn bound}
    For all $m, n \in \mathbb{Z}$, $|mf(n) - nf(m)| < (|m| + |n|)C$.
\end{corollary}

\begin{proof}
    By Lemma \ref{linear bound}, we have that 
    \begin{align*}
        |f(mn) - mf(n)| &< |m|C \\
        |f(mn) - nf(m)| &< |n|C
    \end{align*}
    Then adding the two inequalities and using Triangle Inequality gives \[ |mf(n) - nf(m)| < (|m| + |n|)C\] as desired.
\end{proof}

\begin{theorem}
    The ring of near-endomorphisms $\mathbb{E}$ is isomorphic to $\mathbb{R}$.
\end{theorem}

\begin{proof}

Let $N$ be the set of all near-endomorphisms and let $f \in N$. From these results, we can consider $f$ over $\mathbb{N}$ and find that for any $n, m \in \mathbb{N}$, we have that $|mf(n) - nf(m)| < (|m| + |n|)C$. Dividing by $mn$ on both sides, we observe
\[ \left| \frac{f(n)}{n} - \frac{f(m)}{m} \right| < \left( \frac{1}{n} + \frac{1}{m} \right)C\]

Hence, the sequence defined by $\{\frac{f(n)}{n}\}$ for all $n \in \mathbb{N}$ must be a Cauchy sequence by the above inequality. By the property of Cauchy sequences, it must converge to some real number; let \[\lambda (f) := \lim_{n \to \infty} \frac{f(n)}{n}.\]

We claim that $\lambda$ is an isomorphism of additive groups, and then show the multiplicative property of $\lambda$ to conclude that $\mathbb{E} \cong \mathbb{R}$.

We verify that $\lambda$ is indeed additive: for near-endomorphisms $f, g$, \[\lambda(f+g) = \lim_{n \to \infty} \frac{f(n)}{n} + \frac{g(n)}{n} = \lambda(f) + \lambda(g)\] by the additive properties of limits. Additionally, we observe that the function $f_a (x) = \lfloor ax \rfloor$ satisfies $\lambda(f_a) = a$ for any $a \in \mathbb{R}$, so $\lambda$ is also surjective. 

Next, we show that the kernel of $\lambda$ is the subgroup $B$ of bounded functions. To do this, consider that for $n \in \mathbb{N}$, we have
\begin{align*}
    \lim_{m \to \infty} \left| \frac{f(n)}{n} - \frac{f(m)}{m} \right| &\leq \lim_{m \to \infty} \left(\frac{1}{n} + \frac{1}{m} \right)C \\
    \left|\frac{f(n)}{n} - \lambda(f) \right| &\leq \frac{C}{n} \\
    | f(n) - n\lambda(f) | &\leq C
\end{align*}

With $\lambda(f) = 0$, we observe that $|f(n)| \leq C$, so the values of $f(n)$ are bounded for all $n \in \mathbb{N}$. Additionally, since $f$ is a near-endomorphism, we have that 
\begin{align*}
    | f(0) - f(n) - f(-n) | &< C \\
    |f(-n)| &< C + |f(0)| + |f(n)| \\
    |f(-n)| &< 2C + |f(0)|
\end{align*}
Since $f(0)$ has a finite value, $f(-n)$ must be bounded for all $n \in \mathbb{N}$. As such, $f(n)$ is bounded for all $n \in \mathbb{Z}$; thus, any near-endomorphism mapped to $0$ by $\lambda$ is a bounded function. Additionally, any bounded function $f$ satisfies $\lim_{n \to \infty} \frac{f(n)}{n} = 0$, so all bounded functions are mapped to $0$ by $\lambda$, so $\ker \lambda$ is precisely all bounded near-endomorphisms, as claimed.

Hence, we have that $\lambda$ is an isomorphism of $\lambda: N / B \to \mathbb{R}$. Since $\mathbb{E}$ is precisely $N / B$, we have that $\mathbb{E} \cong \mathbb{R}$ as additive groups.

We finally verify that $\lambda$ is multiplicative under the composition of functions. For $f, g \in N$, we consider three cases: $g$ is bounded, $\lambda(g) > 0$, or $\lambda(g) < 0$. If $g$ is bounded, then $f \circ g$ has finitely many values in its domain and therefore we know that $g, f \circ g \in B$. Thus, $\lambda (f \circ g) = \lambda(f) \lambda(g) = 0$ as we desired.

Thus, we consider the case where $g \in B$ and thus $\lambda(g) \neq 0$. If $\lambda(g) > 0$, then we have 
\begin{align*}
    |g(n) - n\lambda(g)| &\leq C \\
    \lim_{n \to \infty} |g(n) - n\lambda(g)| &\leq C \\
    \lim_{n \to \infty} g(n) &= +\infty
\end{align*}
Hence, we can find that 
\begin{align*}
    \lim_{n \to \infty} \frac{f(g(n))}{g(n)} &= \lim_{n \to \infty} \frac{f(n)}{n} = \lambda(f) \\
    \lambda (f \circ g) &= \lim_{n \to \infty} \frac{f(g(n))}{g(n)} \cdot \frac{g(n)}{n} = \lambda(f) \cdot \lambda(g)
\end{align*}
as we desired. On the other hand, if $\lambda(g) < 0$, then we observe that $\lim_{n \to \infty} g(n) = -\infty$ and that $|f(-n) + g(n)| < |g(0)| + C$. As such, for $n \in \mathbb{N}$
\begin{align*}
    \left|\frac{f(-n)}{n} + \frac{f(n)}{n} \right| &< \frac{|f(0)| + C}{n} \\
    \left| \lim_{n \to \infty} \frac{f(-n)}{n} + \lim_{n \to \infty} \frac{f(n)}{n} \right| &\leq 0 \\
    \lim_{n \to \infty} \frac{f(-n)}{n} &= -\lambda(f)
\end{align*}
This implies
\[\lambda (f \circ g) = \lim_{n \to \infty} \frac{f(g(n))}{g(n)} \cdot \frac{g(n)}{n} = \lim_{n \to \infty} -1 \cdot \frac{f(-n)}{n} \cdot \frac{g(n)}{n} = \lambda(f) \cdot \lambda(g)\]
which shows that $\lambda$ is multiplicative. As such, $\mathbb{E} \cong \mathbb{R}$.
\end{proof}

\section{Constructing the Eudoxus Reals}
\label{construct}

We seek to construct $\mathbb{R}$ using the set of near-endomorphisms of $\mathbb{Z}$ under the operation of pointwise addition and composition as the analogues of addition and multiplication in the reals, respectively. One can easily verify that $\mathbb{E}$ is indeed closed and consistent under these operations.

Under the intuition that the slope of a function corresponds to a real number, we can find that the equivalence class of $\mathbb{E}$ that contains all finite functions corresponds to the additive identity $0$ in $\mathbb{R}$ and the equivalence class that contains the functions sending $f: x \to x$ is the multiplicative identity 1.

Additionally, the associativity of both operations is clearly satisfied. However, while pointwise addition is indeed commutative, one must prove that composition is commutative in $\mathbb{E}$.

\begin{lemma}[Commutativity]
    For near-endomorphisms $f, g$, $f \circ g \sim g \circ f$.
\end{lemma}

\begin{proof}
    By definition of a near-endomorphism, let there exists $C_f, C_g$ such that
    \begin{align*}
        |f(a+b) - f(a) - f(b)| &< C_f \\
        |g(a+b) - g(a) - g(b)| &< C_g
    \end{align*}
    For the sake of simplicity, let $C = \max(C_f, C_g)$.
    
    By Corollary \ref{mn bound}, for any $n$, we can take $m = g(n)$ and $m = f(n)$ respectively to find that
    \begin{align*}
        |nf(g(n)) - g(n)f(n)| &< (|n| + |g(n)|)C \\
        |ng(f(n)) - g(n)f(n)| &< (|n| + |f(n)|)C
    \end{align*}
    and therefore, by triangle inequality, 
    \[|n(f(g(n)) + g(f(n)))| < (2|n| + |f(n)| + |g(n)|)C\]

    We can bound the values of $|f(n)|$ and $|g(n)|$ using Lemma \ref{linear bound} where $a = 1$, finding
    \begin{align*}
        |f(n)| &< |nf(1)| + |n|C \\
        |g(n)| &< |ng(1)| + |n|C
    \end{align*}
    
    Hence, we observe that
    \begin{align*}
        |n(f(g(n)) + g(f(n)))| &< (2|n| + |nf(1)| + |ng(1)| + 2|n|C)C \\
        |f(g(n)) - g(f(n))| &< (2 + |f(1)| + |g(1)| + 2C)C
    \end{align*}

    Since $f(1), g(1), C$ each have bounded values, we find that $f \circ g - g \circ f$ is finite. By definition, $f \circ g \sim g \circ f$.
\end{proof}

The commutativity of composition allows for one to prove the distributive law, and therefore show that $\mathbb{E}$ is a ring.

Furthermore, $\mathbb{E}$ is an ordered ring. Intuitively, we want to define the positive elements to be the ones with positive slope. 
\begin{defn}
    A near-endomorphism $f$ is positive if for any $C > 0$, there exists some $N$ where $f(n) > C$ for all $n > N$.
\end{defn}

\begin{remark}
    An equivalent definition is that a near-endomorphism $f$ is positive if $\{ f(n) : n \in N \}$ has infinitely many positive values.
\end{remark}

It is straightforward to verify that the positive near-endomorphisms are closed under addition and multiplication. 

\begin{lemma}[Trichotomy]\cite{arthan} For any near-endomorphism $f$, exactly one of the following is true: $f$ is bounded, $f$ is positive, or $-f$ is positive.
\end{lemma}

With trichotomy, $\mathbb{E}$ becomes an ordered ring. To prove that $\mathbb{E}$ is a field, we show that for any near-endomorphism $f$, there exists near-endomorphism $g$ such that $f \circ g$ is finite.

\begin{lemma}
\label{neg}
    If $f$ satisfies the properties that $f(x) = -f(-x)$ and $f(a+b) - f(a) - f(b)$ takes on finitely many values for $a, b \in \mathbb{N}$, then $f$ is a near-endomorphism.
\end{lemma}

\begin{proof}
    Since $f$ satisfies the definition of a near-endomorphism for when $a, b \in \mathbb{N}$, we consider two cases: both $a$ and $b$ are negative, or only one of the two is negative.

    In the former case, we observe that $f(a+b) - f(a) - f(b) = -(f(-a-b) - f(-a) - f(-b)$, where $-a, -b \in \mathbb{N}$. Hence, when $a,b$ are both negative, then $f(a+b) - f(a) - f(b)$ can only take on the negative values of $f(a+b) - f(a) - f(b)$ where $a, b \in \mathbb{N}$.

    In the latter case, without loss of generality, let $a$ be negative and $b$ nonnegative. If $a + b \leq 0$, then let $n = b$ and $m = -a - b$. Then, we notice that $m, n$ are both positive. Hence, $f(m+n) - f(m) - f(n) = f(-a) - f(-a-b) - f(b) = f(a+b) - f(a) - f(b)$ reduces this case to when $a, b$ are both positive. Similarly, if $a + b > 0$, then we choose $m = -b$ and $n = a+b$, where $m$ is negative, $n$ is positive, and $m + n < 0$. This gives $f(m+n) - f(m) - f(n) = f(a) - f(-b) - f(a+b) = -(f(a+b) - f(a) - f(b))$. We know $f(m+n) -f(m) - f(n)$ is bounded as this is the case where $a$ is negative, $b$ is positive, and $a + b < 0$; thus, we are done.
\end{proof}

\begin{lemma}[Multiplicative Inverse]
    For every unbounded near-endomorphism $f$, there exists an inverse near-endomorphism $g$ such that $f \circ g$ is a finite function.
\end{lemma}

\begin{proof}
    We wish to be able to construct $g$ in the same way that inverse functions are typically constructed with flipping functions over the line $y=x$, mending the fact that $f$ is not necessarily a one-to-one function. Suppose that $f$ is positive. To do so, for any $x \mathbb{N}$, we choose $g(x) = y \in \mathbb{N}$ as the smallest value satisfying $f(y) \geq x$. This will always be well-defined as $f$ is positive.
    
    Hence, we define the function $g$ in the following way: 
    \[ g(x) = \begin{cases}
        \min(y \in \mathbb{N}, f(y) \geq x) & \text{if } x \geq 0 \\
        -g(-x) & \text{if } x < 0
    \end{cases} \]

    We first verify that $g$ is, in fact, a near-endomorphism. That is, there exists $C$ such that
    \[|g(a+b) - g(a) - g(b)| < C\]

    Consider $a, b \in \mathbb{N}$. Then, $g(a+b), g(a), g(b) \geq 0$. Due to how $g$ was constructed, we also observe that 
    \begin{align*}
        f(g(a)) &\geq  a > f(g(a) - 1) \\
        f(g(b)) &\geq b > f(g(b) - 1) \\
        f(g(a+b)) &\geq a+b > f(g(a+b)-1)
    \end{align*}

    Then, we find that 
    \begin{align*}
        f(g(a+b)) - f(g(a) - 1) - f(g(b) - 1) &> (a+b) - a - b = 0 \\
        f(g(a+b) - 1) - f(g(a)) - f(g(b)) &< (a+b) - a - b = 0
    \end{align*}

    We note that the difference between $f(g(a+b) - g(a) - g(b))$ and the left hand side of both of the above inequalities is bounded. We can find this due to the fact that
    \begin{align*}
        |f(g(a+b)) - f(g(a+b) - g(a) - g(b)) - f(g(a) + g(b))| &< C \\
        |f(g(a)+g(b)) - f(g(a)) - f(g(b))| &< C
    \end{align*}

    and that $|f(x+1) - f(x)| < C + |f(1)|$, which allows us to replace $f(g(a))$ with $f(g(a)-1)$ (and similarly for $g(b)$ and $g(a+b)$) in the inequalities while maintaining that the expression is bounded. Hence, we find that $f(g(a+b) - g(a) - g(b))$ must be bounded; otherwise, since the left-hand side of the inequalities are positive and negative, respectively, the difference between $f(g(a+b) - g(a) - g(b))$ would not be bounded.

    Then, since $f$ is a positive function, then it is impossible for $g(a+b) - g(a) - g(b)$ to take on infinitely many values. If the expression is unbounded above, then $f$ being positive implies that $f(g(a+b) - g(a) - g(b))$ is unbounded. If the expression is unbounded below, then we can come to the same conclusion as $f(g(a+b) - g(a) - g(b)) + f(g(a) + g(b) - g(a+b))$ is bounded. Hence, $g(a+b) - g(a) - g(b)$ takes on finitely many values for $a, b \in \mathbb{N}$. By Lemma \ref{neg}, $g$ is a near-endomorphism.
    
    Then, for sufficiently large $x$, we have that 
    \begin{align*}
        f(g(x)) &\geq x \geq f(g(x) - 1) \geq f(g(x)) - C - |f(1)| \\
        f \circ g - x &\geq 0 \geq f \circ g - C - |f(1)|
    \end{align*}

    so $f \circ g - x$ is a bounded function, and thus $f \circ g \sim x$. Since $x$ can be seen to be the identity function, $g$ is the multiplicative inverse of $f$.

    For negative near-endomorphisms $f$, we can construct the multiplicative inverse $g$ to $f \circ h$, where $h(x) = -x$. Then, $(f \circ h) \circ g = f \circ (h \circ g) = 1$, so $h \circ g$ is the multiplicative inverse of $f$.
\end{proof}

Now that we know inverses exist, we find that $\mathbb{E}$ is an ordered field.

\begin{lemma}[Completeness]
    \cite{arthan}
    $\mathbb{E}$ is complete. For any non-empty, bounded above subset $S \subset \mathbb{E}$, there exists some $s \in \mathbb{E}$ such that for any $x \in S$, the following is true: (1) $x \le s$ and (2) for any $y \in \mathbb{E}$, if $x \le y$ for all $x \in S$, then $s \le y$.
\end{lemma}

Hence, with these lemmas, we know that $\mathbb{E}$ is a complete ordered field. It is known that any two complete ordered fields are isomorphic to one another \cite{arthan}, so $\mathbb{E} \cong \mathbb{R}$.

\section{Uncountablility}
\label{uncountability}

Using the construction of $\mathbb{R}$ through the Eudoxus Reals, we can prove properties of the real numbers.

In particular, we consider a proof of the uncountability of $\mathbb{R}$ by adapting Cantor's Diagonal Argument to near-endomorphisms of $\mathbb{Z}$. To do so, we consider a correspondence between equivalence classes of near-endomorphisms and sequences of integers. This correspondence is analogous to the relationship between real numbers and continued fractions, and, indeed, we use the convergents of continued fractions in this proof.

\begin{lemma}
    For any near-endomorphism $f$, there is a unique integer $a$ such that $f = ax + g$, where $g$ is a near-endomorphism and $0 \leq g < x.$
\end{lemma}

\begin{proof}
    Consider the set of near-endomorphisms $bx$ for $b\in \mathbb{Z}$ that are less than or equal to $f$, for near endomorphism $f.$ There must be a largest such integer $b,$ which we will call $a$.

    Since $ax < f$, we know that $f - ax > 0$. On the other hand, if $f - ax \leq x$, then $f - (a + 1)x \leq 0$, which would contradict $a$'s maximality. Therefore, $0 \leq f - ax < x.$

   Assume for the sake of contradiction that $a$ is not unique. Then there exists an integer $c \neq a$ such that $0 \leq f - cx < x$. If $c > a$, then $f - cx < 0$ because of $a$'s maximality. If $c < a$, then $f - cx = f - ax + (a - c)x \geq (a - x)x \geq x.$ Therefore, by contradiction, $a$ is the unique.
\end{proof}

\begin{defn}
    Suppose $f$ is a near endomorphism. If $a$ is the unique integer such that $0 \leq f - ax < x$, then define the integer part of $f$, $I(f) = a$.
\end{defn}

Now, we will define the following sequences:

\begin{defn}
    For near-endomorphism $f$, define the sequence of integers $A_f$ and the sequence of near endomorphisms $F$ in the following way:

    Let $f_1 = f$, $a_1 = I(f).$
    Then, for $n \leq 1$, let $f_{n + 1} = (f_n - a_nx)^{-1}$, and $a_{n + 1} = f_{n + 1},$ if $(f_n - a_nx)^{-1}$ is defined. Otherwise, both $A_f$ and $F$ will terminate.
\end{defn}

\begin{lemma}
    If $f \sim g$, then $ A_f = A_g$
\end{lemma}

\begin{proof}
    If for near endomorphisms $f, g$, we have $f \sim g$, then $f - g \sim 0$. By definition, $0 \leq g - I(g)x < x$, and adding $f - g \sim 0$ to each part of this inequality, $0 \leq f - I(g)x < x$. But by the uniqueness of $I(f)$, it must be true that $I(f) = I(g)$. Therefore, $A_f$ and $A_g$ have the same $a_1$ term.

    Assume for the sake of induction that $A_f$ and $A_g$ have the same $a_n$ term. Furthermore, assume that $f_n \sim g_n.$ Then, $f_{n+1} = (f_n - a_n)^{-1} \sim (g_n - a_n)^{-1} = g_{n+1}$, or both will be undefined and the sequences $A_f$ and $A_g$ both terminate.

    If $f_{n+1}$ and $g_{n+1}$ are both defined, since they are equivalent, $a_{n+1} = I(f_{n+1}) = I(g_{n + 1})$ will be the same for $A_f$ and $A_g$.

    Therefore, by induction, $A_f = A_g.$
    
\end{proof}

Thus, we have seen that every equivalence class $f$ of near endomorphisms corresponds to a unique finite or infinite sequence of integers $A_f.$ Now we will consider the reverse.

\begin{lemma}
    Every infinite sequence of positive integers $A_f$ corresponds to an equivalence class of near endmorphisms, $f$.
\end{lemma}

\begin{proof}
    For the purposes of this proof, $\frac{a}{b}$ will be used to represent integer division. That is, $\frac{a}{b}$ will evaluate to the unique integer $q$ such that $a = bq + r$, where $0 \leq r < b.$\\

    Suppose the set $A_f = {a_1, a_2, a_3, \cdots}$. Let $P_n, Q_n \in \mathbb{N}$ represent the numerator and the denominator of the $n$th convergent of the continued fraction $[a_1, a_2, a_3, \cdots].$ Define the piecewise function $f = \frac{P_n\cdot x}{Q_n}$ where $Q_n \leq x < Q_{n + 1}.$ We can show that $f$ is a near endomorphism of the natural numbers.

    For $x, y \in \mathbb{N}$, suppose $Q_a \leq  x < Q_{a + 1}$, $Q_b \leq  x < Q_{b + 1}$, and $Q_c \leq  x + y < Q_{c + 1}.$ Note that $c \geq a, b.$ Then $f(x + y) - f(x) - f(y) = \frac{(x + y)P_c}{Q_c} - \frac{xP_a}{Q_a} - \frac{yP_b}{Q_b}.$

    However, $\frac{(x + y)P_c}{Q_c} = \frac{xP_c}{Q_c} + \frac{yP_c}{Q_c} + \delta$, where $\delta = 0$ or $1$. Therefore $f(x + y) - f(x) - f(y) = \left( \frac{xP_c}{Q_c} - \frac{xP_a}{Q_a} \right) + \left(  \frac{yP_c}{Q_c} - \frac{yP_b}{Q_b}\right) + \delta.$

   By the triangle inequality, $\abs{f(x + y) - f(x) - f(y)} \leq \abs{\frac{xP_c}{Q_c} - \frac{xP_a}{Q_a}} + \abs{\frac{yP_c}{Q_c} - \frac{yP_b}{Q_b}} + \delta.$ Because $c \leq a, b$ and continued fractions convergents alternate, the convergent farthest away from the $n$th convergent of a continued fraction is the $n + 1$st convergent, of the convergents past the \textit{n}th convergent. Therefore, this expression is less than or equal to $\abs{\frac{xP_{a+1}}{Q_{a+1}} - \frac{xP_a}{Q_a}} + \abs{\frac{yP_{b+1}}{Q_{b+1}} - \frac{yP_b}{Q_b}} + \delta.$

   This can be rewritten as $\abs{\frac{xP_{a+1}Q_a}{Q_{a+1}Q_a} - \frac{xP_aQ_{a+1}}{Q_aQ_{a+1}}} + \abs{\frac{yP_{b+1}Q_b}{Q_{b+1}Q_b} - \frac{yP_bQ_{b+1}}{Q_bQ_{b+1}}} + \delta.$ 
   
   This is equal to $\abs{\frac{xP_{a+1}Q_a - xP_aQ_{a+1}}{Q_{a+1}Q_{a}} + \delta_1} + \abs{\frac{yP_{b+1}Q_b - yP_bQ_{b+1}}{Q_{b+1}Q_b}+ \delta_2} + \delta,$ where $\delta_1, \delta_2 = 0, -1$.

   Since $P_{a+1}Q_a - P_aQ_{a+1} = (-1)^{a+1}$, and likewise for $b$, and since the absolute values get rid of the negative ones, this is equal to. $\abs{\frac{x}{Q_{a+1}Q_{a}} + \delta_1} + \abs{\frac{y}{Q_{b+1}Q_b}+ \delta_2} + \delta$.

   Lastly, since $x < Q_{a+1}$ and $y < Q_{b + 1},$ this is less than $\abs{\frac{1}{Q_{a}} + \delta_1} + \abs{\frac{1}{Q_b}+ \delta_2} + \delta$, which is indeed bounded since it becomes smaller as $a, b$ get bigger. Therefore, since $|f(x+y) - f(x) - f(y)|$ was bounded by this quantity, $f$ must be a near endomorphism over the natural numbers.

   Furthermore, according to Arthan \cite{arthan}, defining $f(x) = -f(-x)$ for $x < 0$ and $f(0) = 0$ will ensure that $f$ is a near endomorphism over the integers.\\

   Next, consider the fact that $(\frac{xP_n}{Q_n})^{-1} = \frac{xQ_n}{P_n}$. We know this because $\frac{xP_n}{Q_n} = y$ is the number such that $xP_n = yQ_n + r$, where $0 \leq r < Q_n.$ Plugging \textit{y} into $\frac{xQ_n}{P_n}$, we get $\frac{yQ_n}{P_n} = \frac{xP_n - r}{P_n}$. Since $r < Q_n \leq P_n$, this evaluates to $x - 1 \sim x$.

   Similarly, we can define the piecewise function $g(x) = \frac{xQ_n}{P_n}$ for $P_n \leq x < P_{n+1}.$ This function is the inverse of $f = {P_n\cdot x}{Q_n}$ where $Q_n \leq x < Q_{n + 1}.$ We can see that the range of $f(x)$ corresponding with the domain restriction $Q_n \leq x < Q_{n + 1}$ matches up with the domain restriction on $g:$ $P_n \leq x < P_{n+1}:$

   $Q_n \leq x < Q_{n + 1} \implies P_n \leq x < \frac{Q_{n+1}P_n}{Q_n}$. But $P_{n+1}Q_n - Q_{n+1}P_n = (-1)^{n+1} \implies Q_{n+1}P_n = P_{n+1}Q_n - (-1)^{n+1}.$ So $x < \frac{Q_{n+1}P_n}{Q_n} = \frac{P_{n+1}Q_n - (-1)^{n+1}}{Q_n} < P_{n+1}$. Therefore, $P_n \leq x < P_{n+1}.$\\

   Now, we will show that the sequence of integers associated with $f(x)$ is indeed $A_f:$
   Note firstly that $I(f) = a_1$, the first element of the sequence of integers that makes up the continued fractions used to construct $f.$ We will show this below:
   
   $f(Q_n) = P_n$ for all of the infinitely many convergents of $[a_1, a_2, a_3, \cdots].$ For these points, $\frac{f(x)}{x} = a_1.$

   Assume for the sake of contradiction that $I(f) > a_1$. Then, $f(x) - xI(f) < 0$ for infinitely many \textit{x}: namely, all the $x = Q_n.$ By the definition of positive, this would mean that $f(x) - xI(f)$ is not greater than or equal to 0, which contradicts the definition of $I(f).$

   Now, assume for the sake of contradiction that $I(f) < a_1.$ Then $f(x) - xI(f) \geq x$ for the infinitely many values $x = Q_n.$ Therefore, there are infinitely many values where $f(Q_n) - Q_nI(f) - Q_n\geq 0$, meaning that $f(x) - xI(f) - x$ is positive by definition, and therefore, $f(x) - xI(f) \geq x,$ also contradicting the definition of $I(f).$ Thus, $I(f) = a_1.$

   But $f(x) - xI(f) = \frac{xP_n}{Q_n} - a_1x \equiv \frac{x(P_n - a_1Q_n)}{Q_n}.$ This just equals $\frac{xP_n'}{Q_n'}$, where $\frac{P_n}{Q_n}$ are the convergents of the continued fraction $[0, a_2, a_3, a_4, \cdots].$
   
   According to the inverse rule we found earlier, $f_2 = (f(x) - xI(f))^{-1} = \frac{xQ_n'}{P_n'}$, but $\frac{Q_n'}{P_n'}$ are just the reciprocals of the convergents of the continued fraction $0 + \frac{1}{a_2 + \frac{1}{a_3 + \cdots}}$, which are just the convergents of the continued fraction $[a_2, a_3, a_4, \cdots].$

   Therefore, just as the first terms in our sequence of numbers was $I(f) = a_1$, the second term is $I(f_2) = a_2.$ It can be shown through induction that this process continues, as the \textit{n}th term in the sequence produced by $f(x)$ is $I(f_n) = a_n.$ Therefore, $f(x)$ produces the sequence of integers $A_f.$ Furthermore, there exists an equivalence class of near endomorphisms that produces $A_f$: namely, the one containing $f.$

   \end{proof}

   With these lemmas, we may adapt Cantor's diagonalization argument to prove that there are uncountably many real numbers using the Eudoxus construction.

   \begin{theorem}
       There are uncountably many real numbers.
   \end{theorem}

   \begin{proof}
        Assume for the sake of contradiction that there are countably many equivalence classes of near endomorphisms. Therefore, we can order them into a sequence $r(1), r(2), r(3), \cdots$. Each of these real numbers may be expressed as a finite or infinite sequence of integers, $A_{r(n)}.$

        Now, consider the sequence of numbers $A,$ such that $a_n = |a_{r(n), n}| + 1$ if $a_{r(n), n}$, the $nth$ element of $A_{r(n)}$ is defined, or $a_n = 1$ otherwise.

        Since $A$ is a sequence of natural numbers, there exists an equivalence class of near endomorphisms that produces $A.$ Meanwhile, none of the equivalence classes of near endomorphisms $r(1), r(2), \cdots$ could have produced A because they all produce a unique sequence $A_{r(n)}.$ $A_{r(n)}$ and $A$ cannot be the same sequence of integers because, when $A_{r(n)}$'s \textit{n}th term is defined, it is unequal to the \textit{n}th term of $A.$

        Therefore, the equivalence class of near endomorphisms that produced $A$ cannot have been included in our list $r(1), r(2), \cdots,$ which is a contradiction because the list included all the equivalence classes of near endomorphisms. Since we have reached a contradiction, our assumption must have been false, and there are uncountably many equivalence classes of near endomorphisms.

        Since $\mathbb{E} \cong \mathbb{R},$ the real numbers are uncountably infinite.
       
   \end{proof}

\section{Generalizing the Eudoxus Reals Construction}
\label{generalize}

We generalize the above construction and determine the ring of quasi-endomorphisms of other abelian groups. Just like the construction of the reals using the integers, this construction has the advantage of being able to create relatively complicated objects out of much simpler ones.

Following Hermans \cite{hermans}, we define the following, which are a natural generalization of the definitions over $\mathbb{Z}$.

\begin{defn}
    Let $A$ and $B$ be two abelian groups. Define a function $f: A \to B$ be an \textit{almost-homomorphism} if $\{f(a+b) - f(a) - f(b) \}$ is finite. Denote the set of almost-homomorphisms from $A$ to $B$ by $\AHom(A, B)$.
\end{defn}

Here, the requirement that $f(a+b)-f(a)-f(b)$ is finite is a natural generalization of the requirement for a near-endomorphism of $\mathbb Z$ that $|f(a+b)-f(a)-f(b)|$ is bounded, as the former condition is equivalent to the latter when considering the abelian group as $\mathbb Z$. Similarly, we generalize the notion of a bounded function.

\begin{defn}
    Let $A$ and $B$ be two abelian groups. Define a function $f: A \to B$ to be \textit{almost-zero} if $\{ f(a) | a\in A\}$ is finite. Denote the set of almost-zero functions from $A$ to $B$ by $\Az(A, B)$.
\end{defn}

It is not hard to check that $\Az(A, B)$ is a subgroup of $\AHom(A,B)$, where both sets are taken to have the operation of pointwise addition. Since both groups are abelian, we can define the quotient: 

\begin{defn}
    For abelian groups $A$ and $B$, a \textit{quasi-homomorphism} $f: A \to B$ is an element of the quotient group $\AHom(A, B)/\Az(A, B)$. Denote the group of quasi-homomorphisms by $\qhom(A, B)$.
\end{defn}

Following \cite{hermans}, we observe that the abelian groups form a category.

\begin{theorem}\cite{hermans}
    Abelian groups together with the set $\qhom(A, B)$ of morphisms between them form a category \textsf{Qab}, with the natural composition law.
\end{theorem}

Moreover, if we define ``multiplication" of functions by composition, it is not hard to verify that $\qhom(A,B)$ in fact forms a ring, using similar reasoning as for the integers $\mathbb{Z}$. It is often helpful, as in the case for $\mathbb Z$, to consider the ring of quasi-homomorphisms from a group to itself, so we define:

\begin{defn}
    For an abelian group $A$, denote the ring of quasi-homomorphisms $\qhom(A, A)$ by $\qend(A)$.
\end{defn}

With this more general definition, it is natural to ask about the ring of quasi-endomorphisms of various groups. From before, we already know one:

\begin{example}
    The ring of quasi-endomorphisms of the integers $\qend(\mathbb Z)$ is isomorphic to $\mathbb R$.
\end{example}

As another example:

\begin{remark}
    Note that $\qend(A)$ is not interesting if $A$ is a finite group: every function $f: A \to A$ takes on finitely many values and therefore is automatically an almost-homomorphism. Furthermore, it is also almost-zero, so all functions are in fact in the same equivalence class, implying that $\qend(A) \cong \{0\}$.
\end{remark}

\section{The Ring of Quasi-Endomorphisms of \texorpdfstring{$S^{-1}\mathbb{Z}/\mathbb{Z}$}{YP17}}
\label{local}

Hermans \cite{hermans} has previously proven the following.

\begin{theorem}\cite{hermans}\label{hermansp}
    The ring $\qend(\mathbb{Z}[\frac{1}{p}]/\mathbb{Z})$ is isomorphic to the ring of $p$-adic numbers $\mathbb Q_p$.
\end{theorem}

Using the fact that $\mathbb{Q}/\mathbb{Z} \cong \oplus_{p \text{ prime}} \mathbb{Z}[\frac{1}{p}]/\mathbb Z$, Hermans also proves the following.

\begin{theorem}\cite{hermans}
    The ring $\qend(\mathbb{Q}/\mathbb{Z})$ is isomorphic to the finite adele ring of the rationals, namely \[  \sideset{}{'}\prod_{p \text{ prime}} (\mathbb{Q}_p, \mathbb{Z}_p) := \mathbb{A}_{\mathbb{Q}}^{\text{fin}}.\]

    Furthermore, \[\qend(\mathbb Q) \cong \mathbb{A}_{\mathbb{Q}} := \mathbb{R} \times \sideset{}{'}\prod_{p \text{ prime}} (\mathbb{Q}_p, \mathbb{Z}_p)\]
\end{theorem}

We generalize these results by considering $S^{-1}\mathbb Z/\mathbb Z$ for an arbitrary multiplicatively closed subset $S$ of $\mathbb Z$.

\begin{defn}
    Let $S$ be a multiplicatively closed subset of $\mathbb Z$. Consider the set $\mathbb Z \times S$ under the equivalence relation of $(n, s_1) \sim (m, s_2)$ if and only if there exists $s \in S$ such that $s(ns_2 - ms_1) = 0$. This set forms a ring, the \textit{localization of} $\mathbb{Z}$ \textit{by} $S$ (denoted $S^{-1} \mathbb Z$) under the operations
    \begin{align*}
        (n, s_1) + (m, s_2) &:= (ns_2 + ms_1, s_1 s_2) \\
        (n, s_1) \cdot (m, s_2) &:= (nm, s_1s_2)
    \end{align*}
\end{defn}

We often denote the pair $(n, s)$ by $\frac{n}{s}$, since the operations defined above resemble fraction addition and multiplication.

\begin{remark}
    The above definition of localization holds for an arbitrary commutative ring $R$ in place of $\mathbb Z$ and an arbitrary multiplicatively closed subset $S \subset R$. However, in this paper we will only be concerned with the case $R = \mathbb Z$. In that case, since $\mathbb Z$ is an integral domain, the condition for $(n, s_1) \sim (m, s_1)$ simplifies to $ns_2 - ms_1 = 0$.
\end{remark}

We observe that there are many multiplicatively closed subsets of $\mathbb Z$ which do not have a finite set of generators. An example is all numbers that are $1$ mod $17$, which cannot have a finite set of generators because there are infinitely many primes which are $1$ mod $17$. Moreover, it is not the case that every multiplicatively closed subset of $\mathbb Z$ is generated by some subset of the primes, as in the above example. For example, the set generated by $2p$ for all $p \geq 3$ cannot be generated by some subset of the primes, because it does not contain any prime other than $2$.

However, for the purposes of examining localizations of $\mathbb Z$, the situation is not nearly as bad as the examples above might make it seem. In particular, every localization $S^{-1} \mathbb Z$ is isomorphic to a localization $T^{-1} \mathbb Z$ where $T$ is generated by some subset of the primes and $-1$.

\begin{defn}
    Let $S$ be a multiplicatively closed subset of $\mathbb Z$. Define the \textit{saturation} of $S$, denoted $\hat{S}$, to be \[\hat{S} := \{ r \in \mathbb{Z} : \exists s \in \mathbb Z \text{ s.t. } rs \in S \}\]
\end{defn}

\begin{example}
    Consider the set $S$ mentioned earlier generated by $2p$ for all $p \geq 3$. The saturation of this set must include every prime at least $3$ because for every $p$, there exists $s \in \mathbb Z$, namely $2$, such that $2p \in S$. Moreover, it includes $2$ since $2 \cdot 3 \in S$. Therefore, the saturation of $S$ is just all natural numbers.
\end{example}

\begin{theorem}
\label{saturation}
    For any multiplicatively closed subset $S$ of $\mathbb Z$, \[\hat{S}^{-1} \mathbb Z \cong S^{-1} \mathbb Z.\]
\end{theorem}

\begin{proof}
    It is clear that $\hat{S}$ is at least as large as $S$, because given $s \in S$, there exists some element of $\mathbb Z$, namely $1$, such that $s \cdot 1 \in S$. So $S^{-1} \mathbb Z / \mathbb Z$ is a subring of $\hat{S}^{-1} \mathbb Z / \mathbb Z$. For the other direction, observe that for all $r \in \hat{S}$, there exists $s \in \mathbb Z$ such that $rs \in S$, by definition. Then $\frac{s}{rs} \in S^{-1} \mathbb Z/\mathbb Z$, but using the equivalence relation this is equivalent to $\frac{1}{r}$. So $\frac{1}{r} \in S^{-1} \mathbb Z/\mathbb Z$ for all $r \in \hat{S}$, which implies $S^{-1} \mathbb Z/\mathbb Z \supset \hat{S}^{-1} \mathbb Z / \mathbb Z$ as $1$ generates $\mathbb Z$.
\end{proof}

In other words, to find $\qend(S^{-1}\mathbb Z/\mathbb Z)$, it suffices to consider the saturations of every multiplicatvely closed subset of $\mathbb Z$.

\begin{lemma}
    The saturation of every multiplicatively closed subset of $\mathbb Z$ is generated by some subset of the primes and $-1$.
\end{lemma}
\begin{proof}
    Consider a set of generators of the multiplicatively closed subset $S$. If any generator is negative, then using Theorem \ref{saturation}, $-1 \in \hat{S}$. For every generator $g$, by Theorem \ref{saturation} any prime $p$ dividing $g$ must be in $\hat{S}$. Moreover, the set generated by every prime dividing $g$ must contain $g$. Hence the set containing every prime that divides some generator generates $\hat{S}$, and is contained in $\hat{S}$. So it equals $\hat{S}$, which proves the claim. 
\end{proof}

As another simplification, we observe that if $S$ is a saturation of some set which includes $-1$ in its generators, then $S = T \cup (-T)$ where $T = S \cap \mathbb N$. We claim \[T^{-1} \mathbb{Z} \cong S^{-1} \mathbb{Z}.\] This is true because every negative element of $S^{-1} \mathbb{Z}$, say $\frac{-s}{n}$, is equivalent to an element where the numerator is positive, namely $\frac{s}{-n}$, which is in $T^{-1} \mathbb Z$. So in general it suffices to consider sets $S$ which are generated by some subset of the primes. From now on we only consider these sorts of sets. \newline

Motivated by the techniques in \cite{hermans}, we split up any finitely generated $S$ and relate the rings of quasi-homomorphisms.

\begin{proposition}
\label{crt}
    Let $\{ p_1, \ldots, p_k \}$ and $\{q_1, \ldots, q_\ell \}$ be two sets of distinct primes that do not overlap. Then \[\mathbb{Z}\left[\frac{1}{p_1}, \ldots, \frac{1}{p_k}\right]\Big/\mathbb{Z} \times \mathbb{Z}\left[\frac{1}{q_1}, \ldots, \frac{1}{q_\ell}\right]\Big/\mathbb{Z} \cong \mathbb{Z}\left[\frac{1}{p_1}, \ldots, \frac{1}{p_k}, \frac{1}{q_1}, \ldots, \frac{1}{q_\ell}\right]\Big/\mathbb{Z}\] as groups.
\end{proposition}
\begin{proof}
    Define the map
    \begin{align*}
        \phi: \mathbb{Z}\left[\frac{1}{p_1}, \ldots, \frac{1}{p_k}\right]\Big/\mathbb{Z} \times \mathbb{Z}\left[\frac{1}{q_1}, \ldots, \frac{1}{q_\ell}\right]\Big/\mathbb{Z} &\to \mathbb{Z}\left[\frac{1}{p_1}, \ldots, \frac{1}{p_k}, \frac{1}{q_1}, \ldots, \frac{1}{q_\ell}\right]\Big/\mathbb{Z} \\
        \left(\frac{a}{n}, \frac{b}{m} \right) &\to \frac{x}{nm}
    \end{align*}
    where $x$ satisfies $x \equiv a \bmod{n}$ and $x \equiv b \bmod{m}$. We know there exists a unique solution to this system of congruences since $n$ is the product of powers of $p_i$ and $m$ is the product of powers of $q_i$, so they are relatively prime, so applying CRT guarantees a unique solution. To show this map is a homomorphism, consider two arbitrary elements of $\mathbb{Z}\left[\frac{1}{p_1}, \ldots, \frac{1}{p_k}\right]\Big/\mathbb{Z} \times \mathbb{Z}\left[\frac{1}{q_1}, \ldots, \frac{1}{q_\ell}\right]\Big/\mathbb{Z}$, and without loss of generality we can assume the denominators of the first components are equal, and the denominators of the second components are equal, by writing the fractions with a common denominator. Then \[\phi \left( \frac{a_1}{n}, \frac{b_1}{m} \right) + \phi \left( \frac{a_2}{n}, \frac{b_2}{m} \right) = \frac{x_1}{nm} + \frac{x_2}{nm} = \frac{x_1+x_2}{nm}\] where
    \begin{align*}
        x_1 &\equiv a_1 \bmod{n}, x_1 \equiv b_1 \bmod{m} \\
        x_2 &\equiv a_2 \bmod{n}, x_2 \equiv b_2 \bmod{m}
    \end{align*}

    These imply that $x_1+x_2 \equiv a_1 + a_2 \bmod{n}$ and $x_1+x_2 \equiv b_1 + b_2 \bmod{m}$. Since the numerator of $\phi \left( \frac{a_1+a_2}{n}, \frac{b_1+b_2}{m} \right)$ also satisfies these equations, and these equations have a unique solution by CRT, we have that $\phi$ is a homomorphism.

    To show that this is an isomorphism, we observe that the inverse is given by 
    \begin{align*}
        \phi^{-1}: \mathbb{Z}\left[\frac{1}{p_1}, \ldots, \frac{1}{p_k}, \frac{1}{q_1}, \ldots, \frac{1}{q_\ell}\right]\Big/\mathbb{Z} &\to \mathbb{Z}\left[\frac{1}{p_1}, \ldots, \frac{1}{p_k}\right]\Big/\mathbb{Z} \times \mathbb{Z}\left[\frac{1}{q_1}, \ldots, \frac{1}{q_\ell}\right]\Big/\mathbb{Z} \\
        \frac{x}{nm} &\to \left(\frac{x}{n}, \frac{x}{m} \right)
    \end{align*}
    which is clearly a homomorphism since \[\phi \left( \frac{x_1}{nm} \right) + \phi \left( \frac{x_2}{nm} \right) = \left( \frac{x_1}{n}, \frac{x_1}{m} \right) + \left( \frac{x_2}{n}, \frac{x_2}{m} \right) = \left( \frac{x_1+x_2}{n}, \frac{x_1+x_2}{m} \right)\] By construction, $\phi^{-1}$ is the inverse of $\phi$.
\end{proof}

Hermans \cite{hermans} proves that \textsf{Qab} is an additive category, which results in the following very useful lemma relating the ring of quasi-endomorphisms of a direct sum of groups to the rings of quasi-endomorphisms of each of the groups.

\begin{lemma}\cite{hermans}
\label{combine}
    Let $A, B \in \text{Obj}(\textsf{Qab})$. If \[ \qhom(A, B) = \qhom(B, A) = 0,\] then \[\qend(A \times B) \cong \qend(A) \times \qend(B).\]
\end{lemma}

We now show that the conditions of Lemma \ref{combine} are satisfied.

\begin{lemma}
\label{hom0}
    Let $\{ p_1, \ldots, p_k \}$ and $\{q_1, \ldots, q_\ell \}$ be two sets of distinct primes that do not overlap. Then \begin{align*} \qhom \left(\mathbb{Z}\left[ \frac{1}{p_1}, \ldots, \frac{1}{p_k} \right] \Big/ \mathbb Z, \mathbb{Z}\left[ \frac{1}{q_1}, \ldots, \frac{1}{q_\ell} \right] \Big/ \mathbb Z \right) &= 0 \\
    \qhom \left( \mathbb{Z}\left[ \frac{1}{q_1}, \ldots, \frac{1}{q_\ell} \right] \Big/ \mathbb Z , \mathbb{Z}\left[ \frac{1}{p_1}, \ldots, \frac{1}{p_k} \right] \Big/ \mathbb Z \right) &= 0.\end{align*}
\end{lemma}
\begin{proof}
    Since the $p_i$ and $q_i$ are symmetric, it suffices to prove the first of these statements. Let \[f: \mathbb{Z}\left[ \frac{1}{p_1}, \ldots, \frac{1}{p_k} \right] \Big/ \mathbb Z \to \mathbb{Z}\left[ \frac{1}{q_1}, \ldots, \frac{1}{q_\ell} \right] \Big/ \mathbb Z\] be a quasi-homomorphism. We want to show that $f$ is almost zero, that is, its image is finite.

    Since $f$ is a quasi-homomorphism, the set \[N_f := \left\{ f(a+b)-f(a)-f(b) : a, b \in \mathbb{Z}\left[ \frac{1}{p_1}, \ldots, \frac{1}{p_k} \right] \Big/ \mathbb Z \right\}\] is finite. Since for all $n$, there are at most $n$ equivalence classes in the quotient group $\mathbb{Z}[ \frac{1}{p_1}, \ldots, \frac{1}{p_k} ] \big/ \mathbb Z$ which have a denominator of $n$, this implies that the denominators of the elements of the set are bounded. In particular, there are only finitely many distinct denominators. Let $L$ be the lcm of these denominators.

    Now suppose for the sake of contradiction that $f$ were not almost zero, that is, its image is infinite. This implies that the denominators of the images are unbounded, so there exists $\frac{a}{n} \in \mathbb{Z}[ \frac{1}{p_1}, \ldots, \frac{1}{p_k} ] \big/ \mathbb Z$ such that \[ f \left( \frac{a}{n} \right) = \frac{b}{m} \text{ and } m > L.\] Then since $f$ is a quasi-homomorphism,
    \begin{align*}
        f \left( \frac{2a}{n} \right) - f \left( \frac{a}{n} \right) - f \left( \frac{a}{n} \right) &\in N_f \\
        f \left( \frac{3a}{n} \right) - f \left( \frac{2a}{n} \right) - f \left( \frac{a}{n} \right) &\in N_f \\
        &\hspace{-1.7cm}\vdots \\
        f \left( \frac{na}{n} \right) - f \left( \frac{(n-1)a}{n} \right) - f \left( \frac{a}{n} \right) &\in N_f
    \end{align*}
    Each of the expressions on the left-hand side has a denominator dividing $L$, by the definition of $L$. Hence writing each expression with a denominator of $L$ and adding them together gives a fraction whose denominator divides $L$. In particular, we have that \[ f \left( \frac{na}{n} \right) - n f \left( \frac{a}{n} \right) = f(0) - n f \left( \frac{a}{n} \right) = -n \left( \frac{b}{m} \right) = - \frac{bn}{m}\] must have a denominator dividing $L$. Now recall that $\frac{a}{n} \in \mathbb{Z}[ \frac{1}{p_1}, \ldots, \frac{1}{p_k} ] / \mathbb Z$ and $\frac{b}{m} \in \mathbb{Z} [ \frac{1}{q_1}, \ldots, \frac{1}{q_\ell} ] / \mathbb Z$. So $n$ is the product of powers of $p_i$ and $m$ is the product of powers of $q_i$. Since $p_i$ and $q_i$ do not overlap, $\gcd(n, m) = 1$. So the denominator of $-\frac{bn}{m}$ is the same as that of $\frac{-b}{m}$, which is $m$. By construction, $m>L$, a contradiction.

    Therefore all quasi-homomorphisms from $\mathbb{Z}\left[ \frac{1}{p_1}, \ldots, \frac{1}{p_k} \right] \Big/ \mathbb Z \to \mathbb{Z}\left[ \frac{1}{q_1}, \ldots, \frac{1}{q_\ell} \right] \Big/ \mathbb Z$ are almost-zero, so the set of quasi-homomorphisms is $0$, as desired.
\end{proof}

\begin{remark}
Informally, the proof finds some element whose image which has a large denominator, and then adds it to itself many times to give an element which is eventually an integer, giving a contradiction. This is the same idea as the proof of Proposition 2.9 in \cite{hermans}, except that in this case we use the size of the denominator (similar to the $p$-adic absolute value) as a measure of how ``large" the image is, rather than absolute value.
\end{remark}

We now prove the main result.

\begin{theorem}
    Let $S$ be a multiplicatively closed subset of $\mathbb Z$ and let $\hat{S}$ denote its saturation. Suppose that $\hat{S}$ is generated by finitely many primes $p_1, p_2, \ldots, p_k$. Then \[\qend \left( S^{-1} \mathbb Z / \mathbb Z \right) \cong \mathbb{Q}_{p_1} \times \mathbb{Q}_{p_2} \times \ldots \times \mathbb{Q}_{p_k}. \]
\end{theorem}
\begin{proof}
From Theorem \ref{saturation} and Proposition \ref{crt}, we have that, as abelian groups,
\begin{align*}
    S^{-1} \mathbb Z / \mathbb Z &\cong \hat{S}^{-1} \mathbb Z / \mathbb Z \\
    &\cong \mathbb{Z} \left[ \frac{1}{p_1} , \ldots, \frac{1}{p_k} \right] \Big/ \mathbb Z \\
    &\cong \mathbb{Z} \left[ \frac{1}{p_1} \right] \Big/ \mathbb Z \times \cdots \times \mathbb{Z} \left[ \frac{1}{p_k} \right] \Big/ \mathbb Z
\end{align*}
Then by Proposition \ref{hom0}, we can apply Lemma \ref{combine} repeatedly, which gives that
\begin{align*}
    \qhom ( S^{-1} \mathbb{Z} / \mathbb{Z} ) &\cong \qhom \left(\mathbb{Z} \left[ \frac{1}{p_1} \right] \Big/ \mathbb Z \times \cdots \times \mathbb{Z} \left[ \frac{1}{p_k} \right] \Big/ \mathbb Z \right) \\
    &\cong \qhom \left(\mathbb{Z} \left[ \frac{1}{p_1} \right] \Big/ \mathbb Z \times \cdots \times \mathbb{Z} \left[ \frac{1}{p_{k-1}} \right] \Big/ \mathbb Z \right) \times \qhom \left( \mathbb{Z} \left[ \frac{1}{p_k} \right] \Big/ \mathbb{Z} \right)\\
    &\cong \dots \\
    &\cong \qhom \left( \mathbb{Z} \left[ \frac{1}{p_1} \right] \Big/ \mathbb{Z} \right) \times \cdots \times \qhom \left( \mathbb{Z} \left[ \frac{1}{p_k} \right] \Big/ \mathbb{Z} \right)
\end{align*}

From Proposition 2.4 in Hermans \cite{hermans}, we have that $\qhom ( \mathbb{Z} [\frac{1}{p_i}]/\mathbb{Z}) \cong Q_{p_i}$, so we conclude \[\qhom ( S^{-1} \mathbb{Z} / \mathbb{Z} ) \cong \mathbb{Q}_{p_1} \times \cdots \times \mathbb{Q}_{p_k} \]
as desired.
\end{proof}

Based on the result in \cite{hermans} that $\qhom \left( \oplus_{p \text{ prime}} \mathbb{Z}[\frac{1}{p}]/\mathbb Z \right) = \sideset{}{'}\prod_{p \text{ prime}} (\mathbb{Q}_p, \mathbb{Z}_p)$, we conjecture the following, which would give a complete characterization of the ring of quasi-homomorphisms of $S^{-1} \mathbb{Z}/\mathbb{Z}$.

\begin{conjecture}
    Let $S$ be a multiplicatively closed subset of $\mathbb Z$ and let $\hat{S}$ denote its saturation. Suppose that $\hat{S}$ is generated by a possibly infinite number of primes $p_1, p_2, \ldots $. Then \[\qend \left( S^{-1} \mathbb Z / \mathbb Z \right) \cong \sideset{}{'}\prod_{p_i} (\mathbb{Q}_{p_i}, \mathbb{Z}_{p_i}). \]
\end{conjecture}

\section{Conclusion and Further Directions}
\label{conclusion}

In this paper, we have explored an alternate construction of the real numbers via functions from $\mathbb Z$ to $\mathbb Z$. We established properties of $\mathbb E$, such as the fact that it is a complete ordered field, and the fact that it is uncountable, without reference to the real numbers. Finally, we generalized the construction of the Eudoxus reals to other abelian groups, thereby construction the $p$-adics and products of $\mathbb{Q}_p$.

One possible generalization of the results in Section \ref{local} is to consider the localization of other rings of integers at prime ideals. One particular example is to consider $\qend (\mathbb{Z}[i, \frac{1}{\pi}]/\mathbb{Z}[i])$, in analogy to $\qend(\mathbb{Z}[\frac{1}{p}]/\mathbb{Z}) \cong \mathbb{Q}_p$. It is conjectured that this is an extension of $\mathbb{Q}_{N(\pi)}$ in the case that $\pi$ is not real, and is the ring of $2$x$2$ matrices over $\mathbb{Q}_{\pi}$ in the case that $\pi \equiv 3 \bmod{4}$ is a real prime. 

Another direction to explore is to find other proofs of important properties by understanding the Eudoxus Reals construction and characterizing the quasi-endomorphisms of the $\mathbb{Z}$. One may work to understand and prove the Heine-Borel Theorem, which helps provide an understanding of the foundations of real analysis.

Finally, one direction to explore is using the close approximations that the convergents of continued fractions provide in order to characterize the set of quasi-endomorphisms and justify the uncountability of the real numbers. Further work can explore and understand the Eudoxus reals using more properties of continued fractions.

\section{Acknowledgements}

We would like to thank Matt Baker for proposing this problem. We would also like to thank our counselor, Yiyang Liu, for his guidance and help.

We thank the PROMYS program and the Clay Mathematical Institute for the opportunity to conduct this research, and John Sim and David Fried for running the research labs. 

\bibliography{citations}

\begin{thebibliography}{1}

\bibitem{arthan}
R.~D. Arthan.
\newblock The eudoxus real numbers, 2004.

\bibitem{hermans}
T.~Hermans.
\newblock An elementary construction of the real numbers, the p-adic numbers,
  and the rational adele ring., June 2018.

\end{thebibliography}
\bibliographystyle{abbrv}

\end{document}